\documentclass{article}

\usepackage{amsfonts}
\usepackage{amsmath}
\usepackage{a4wide}
\usepackage{amssymb}
\usepackage{url}
\usepackage[noend]{algorithmic}
\usepackage{theorem}
\theorembodyfont{\slshape}

\newtheorem{lemma}{Lemma}[section]
\newtheorem{theorem}[lemma]{Theorem}
\newtheorem{corollary}[lemma]{Corollary}
\newtheorem{proposition}[lemma]{Proposition}
{
 \theorembodyfont{\normalfont}
\newtheorem{remark}[lemma]{Remark}

}

\newcommand\bad{{\mathsf{Bad}}}
\newcommand\good{{\mathsf{Good}}}
\newcommand\verybad{{\mathsf{VeryBad}}}

\newenvironment{proof}{\paragraph*{Proof}}
{\par}

\newcommand\qed{\hfill$\square$}

\newcommand\OO{{\mathcal O}}

\newcommand\calF{{\mathcal F}}

\newcommand\gal{{\mathrm{Gal}}}

\newcommand\GL{{\mathrm{GL}}}

\newcommand\SL{{\mathrm{SL}}}

\newcommand\eps\varepsilon
\newcommand\ph\varphi
\newcommand\C{{\mathbb C}}

\newcommand\F{{\mathbb F}}
\newcommand\Q{{\mathbb Q}}

\newcommand\Z{{\mathbb Z}}
\newcommand\height{{\mathrm h}}

\newcommand\HH{{\mathcal H}}
\newcommand\bfa{{\mathbf a}}

\newcommand\spl{{\mathrm{split}}}

\newcommand\PP{{\mathbb P}}

\newcommand\splic{{\mathrm{sp.C}}}

\newcommand{\xrp}{X_0^+(p^r)}
\newcommand{\xr}{X_0(p^r)}
\newcommand{\ok}{{{\mathcal O}_K}}
\newcommand{\jg}{J_0(p)^{\sharp}}
\newcommand{\je}{{J_e^{\sharp}}}

\title{Rational points on $X_0^+ (p^r )$}
 
\author{Yu. Bilu, P. Parent, M. Rebolledo}

1

\begin{document}

\hfuzz=3pt

\maketitle

\begin{abstract}
We show how the recent isogeny bounds due to  Gaudron and R\'emond allow to obtain
the triviality of $X_0^+ (p^r )(\Q )$, for ${r>1}$ and~$p$ a prime exceeding $2\cdot10^{11}$. This includes 
the case of the curves $X_\spl (p)$. We then prove, with the help of computer calculations, that the same 
holds true for~$p$ in the range $11\leq p\leq 10^{14}$, $p\neq 13$. The combination of those results completes the 
qualitative study of such sets of rational points undertook in~\cite{BP11} and~\cite{BP10}, with the exception
of~$p=13$. 
\medskip

AMS 2010 Mathematics Subject Classification  11G18 (primary), 11G05, 11G16 
(secondary). 
\end{abstract}

\begin{flushright}
\textit{To the memory of Fumiyuki Momose}
\end{flushright}

\section{Introduction}
\addtocontents{toc}{\vspace{-0.7\baselineskip}}
For~$p$ a prime number and ${r>1}$ an integer, let~$X_0 (p^r )$ be the usual modular curve parameterizing geometric
isomorphism classes of 
 curves endowed with a cyclic isogeny of degree~$p^r$, and let ${X_0^+ (p^r ):=X_0 (p^r )/w_p}$ 
be its quotient by the Atkin-Lehner involution. When ${r=2s}$ is even, $X_0^+ (p^{2s} )$ is  $\Q$-isomorphic to the modular curve known as~$X_\spl 
(p^s )$. The curves $X_0^+ (p^r )$ have motivated a number of works, dating back at least to Mazur's foundational paper~\cite{Ma77}, where the case of $X_\spl (p)$ 
was tackled. Momose, among others, obtained important results in~\cite{Mo84} and~\cite{Mo86}.
 
In~\cite{BP11,BP10} we proved that for some absolute constant~$p_0$, the only rational points of~${X_0^+ (p^r )}$ with ${p>p_0}$ and ${r>1}$ are trivial, that is, 
the unavoidable cusps and CM points. One easily checks the existence of degeneracy morphisms $X_0^+ (p^{r+2})\to X_0^+ (p^{r} )$ which
show it is sufficient to settle the cases ${r=2}$ and~$3$ (see e.g.~\cite{Mo86}, p. 443). Our method uses three main ingredients: an integrality statement for 
non-cuspidal rational points (Mazur's method), an upper bound for the height of integral points (Runge's method), and a lower bound for the height of rational 
points (isogeny bounds, obtained by the transcendence methods). The combination of those yields inequalities of the following shape for the height of a (non-cuspidal 
and non-CM) rational point~$P$:
\begin{align}
\label{ineq1}
c\, p &< \height (P) < 2\pi \sqrt{p} +O(\log p) & (r=2), \\
\label{ineq2}
c' p^{3/2} &<  \height (P) < 24 p\log p +O(p)& (r=3),
\end{align}
where~$c$ and~$c'$ are positive constants. This of course yields a contradiction when~$p$ exceeds certain~$p_0$, but the value for~$p_0$ in~\cite{BP11,BP10} 
was extremely large, due to the huge size of the constants~$1/c$ and~$1/c'$ furnished by the transcendence theory. 

   In previous works~\cite{Pa05,Re08} we had developed very 
different methods leading to the same triviality results for primes in certain congruence classes. We were not able to make those earlier techniques prove 
triviality of integral points for almost all primes; on the other hand, they are very fit for dealing with small primes~$p$. 

The aim of the present paper is therefore twofold. First we make the above inequalities~(\ref{ineq1}) and~(\ref{ineq2}) completely explicit. We did not try to 
obtain the numerical value of~$p_0$ in \cite{BP11,BP10}, but a calculation shows that in both cases  triviality of $X_0^+ (p^r )(\Q )$ was established for~$p$ 
exceeding $10^{80}$ (which is supposed to be approximately the number of atoms in the visible universe). Now, thanks to the work of Gaudron and  
R\'emond~\cite{GR11}, who obtained drastic numerical improvements of classical isogeny bounds, we can size this down to the much more manageable ${p\ge 
1.4\cdot10^7}$ for ${r=2}$ and ${p>1.7\cdot 10^{11}}$ for ${r=3}$. 

The second aim of this article is then to develop an algorithm based on the Gross vectors method \cite{Pa05,Re08} and to explain how to use it on a computer to 
rule out primes in the range ${11\leq p\leq 10^{14}}$, ${p\neq 13}$. This results in the following theorem.
\begin{theorem} 
\label{MT}
The points of $X_0^+ (p^r )(\Q )$ are trivial for all prime numbers ${p\geq 11}$, ${p\neq 13}$, and all integers ${r>1}$.  
\end{theorem}
It is perhaps worth stressing here that, even if the help of a computer was forced by the important range of primes we had to consider, the computations 
themselves are very elementary, so that it takes only a few minutes to rule out a given prime by hand - even much beyond our bound $10^{14}$. We refer
the skeptical reader to Section~\ref{Gross}.  

For the remaining very small primes our methods break down, but ad hoc studies almost completely cleaned-up the situation, see \cite[Theorem~3.6]{Mo86}, 
\cite[Theorems~0.1 and~3.14]{MS02},   and \cite[Section~10]{Ga02}. Precisely:
\begin{itemize}
\item
for $p=2$ we have  $X_0^+ (2^r ) \simeq \PP^1$ for $2\leq r \leq 5$ (the corresponding curves having thereby infinitely many $\Q$-points) and $X_0^+ (2^r )(\Q)$ is 
trivial for $r\geq 6$;

\item
for $p=3$
 we have $X_0^+ (3^r ) \simeq \PP^1$ for $2\leq r\leq 3$ and $X_0^+ (3^r )(\Q)$ is trivial for $r\geq4$;

\item 
for $p=5$ we have $X_0^+ (5^2 ) \simeq\PP^1$, the curve $X_0^+ (5^3 )$ has one well-described non-trivial $\Q$-point \cite[Section~10]{Ga02} and $X_0^+ (5^r )(\Q)$ 
is trivial for $r\geq 4$;

\item
for $p=7$ we have $X_0^+ (7^2 )\simeq\PP^1$ and $X_0^+ (7^r )(\Q)$ is trivial for $r\geq3$ ;

\item 
for $p=13$ the set $X_0^+ (13^r )(\Q)$ is trivial for $r\geq3$. 

\end{itemize}

The only remaining question mark therefore concerns $X_0^+ (13^2 )\simeq X_\spl (13)$:  this curve has  genus 3 (so only a finite number of rational points) and 
Galbraith~\cite{Ga99} or Baran~\cite{Ba11} spotted seven (trivial) points, which they conjecture exhaust~$X_0^+ (13^2 )(\Q )$, but this still has to be checked\dots  
We continue this discussion of the level $13$ case in Remark~\ref{verysmall}. On the other hand, the question for 
the curves~$X_0^+ (p)$ remains, as far as we know, essentially open, apart from some partial or experimental results (see for instance~\cite{Ga99,Go01}). 
In prime level our methods indeed fail for deep reasons akin to the ones that make the case of~$X_{\mathrm{nonsplit}} (p)$ so difficult (see, for instance, the introduction 
to~\cite{BP11}).  

  The problem of describing points over higher number fields is also extremely open (as it is a fortiori the case for the curves~$X_0 (N)$). As explained 
in~\cite{Bi02,BI11}, one can explicitly bound integral and even $S$-integral points over arbitrary number field using Baker's method, but these bounds are quite huge and not very useful because of lack of integrality results. Finally, our techniques should at least partially extend to curves~$X_0^+ (N)$ where~$N$ has several prime factors (or 
even curves~$X_0 (N)/W$, where $W$ is the full group generated by the Atkin-Lehner involutions, at least in the easier case where~$N$ is not square-free). We plan to 
pursue this study in forthcoming works.

Let us recall two immediate consequences of Theorem~\ref{MT} for the arithmetic of elliptic curves. The first concerns Serre's uniformity problem over~$\Q$ 
\cite{Se72,BP11}. Recall that to an elliptic curve over a field~$K$ and a prime number~$p$ (distinct from the characteristic of~$K$) one associates the Galois 
representation ${\rho_{E,p}:\gal(\bar K/K)\to 
\GL(E[p])\cong \GL_2(\F_p)}$. Serre~\cite{Se72} proved that, given a non-CM elliptic curve~$E$ defined over a number field~$K$, there exists ${p_0=p_0(E,K)}$ such 
that for ${p>p_0}$ the representation $\rho_{E,p}$ is surjective. He asked if~$p_0$ can be made independent of~$E$. In particular, in the case ${K=\Q}$ (which will be 
assumed in the sequel) it is widely believed that ${p_0=37}$ would do: 
\textsl{%
\begin{quotation}
\noindent
let~$E$ be a non-CM elliptic curve over~$\Q$, and ${p>37}$ a prime number; is it true that the associated Galois representation  is surjective? 
\end{quotation}}

As explained in the introduction of~\cite{BP11}, to answer this question affirmatively it suffices to show that the image of the Galois representation is not contained in the normalizer  of a (split or non-split) Cartan subgroup of $\GL_2(\F_p)$. Since elliptic curves over~$\Q$ for which the image of $\rho_{E,p}$ is contained in the normalizer 
of a split Cartan subgroup are parametrized by the $\Q$-points on the curve ${X_\spl(p)\simeq X_0^+(p^2)}$ (see section~\ref{Runge}), Theorem~\ref{MT} has as 
immediate consequence the following improvement of the main result of~\cite{BP11}.

\begin{corollary}
Let~$E$ be an elliptic curve over~$\Q$ without complex multiplication and~$p$ a prime number, ${p\ge 11}$, ${p\ne 13}$. Then the image of the Galois representation 
${\rho_{E,p}:\gal(\bar \Q/\Q)\to  \GL_2(\F_p)}$ is not contained in the normalizer of a split Cartan subgroup of  $\GL_2(\F_p)$. 
\end{corollary}
 
\medskip

Another application of Theorem~\ref{MT} concerns elliptic $\Q$-curves. Recall that an elliptic curve \textsl{with} complex multiplication, defined over~$\bar\Q$,  is 
isogenous to any of its conjugates (over~$\Q$). A \textsl{$\Q$-curve} is an elliptic curve \textsl{without} complex multiplication over $\bar\Q$ with the same 
property, that is, which is isogenous to each of its conjugates over~$\Q$. This notion was first introduced by Gross (in the setting of CM curves) in~\cite{Gr80}; 
for more about this concept we refer in particular to the work of Elkies~\cite{El04}. 

When a $\Q$-curve is \textsl{quadratic} (that is, defined over a quadratic field), we will say that  it \textsl{has degree~$N$} if there is a cyclic $N$-isogeny 
from the curve to its only non-trivial conjugate. For concrete examples of quadratic $\Q$-curves see for instance~\cite{Ga02} and references therein. 

It is known that quadratic $\Q$-curves of degree~$N$ are parametrized by the non-CM rational points of the curve $X_0^+(N)$, see \cite[beginning of Section~7]{BP10}. 
Hence Theorem~\ref{MT} has the following consequence, improving on the main result of~\cite{BP10}.

\begin{corollary}
Let~$p$ be a prime number, ${p\ge 11}$ and ${p\ne 13}$. Then for ${r>1}$ there does not exist quadratic $\Q$-curves of degree~$p^r$. 
\end{corollary}

\paragraph{Plan of the article}
The material is organized as follows. In Section~\ref{Runge} we make the upper bounds in~(\ref{ineq1}) and~(\ref{ineq2}) explicit. In
Section~\ref{Isogeny} we deduce the explicit lower bounds in~(\ref{ineq1}) and~(\ref{ineq2}) from the Gaudron-R\'emond version
of the isogeny theorem. The method and computations for small primes are explained in Section~\ref{Gross}. Let us finally note that, due to 
the nature of our proofs, the cases ${r=2}$ and ${r=3}$ are not completely similar, so we often prefer deal with each case separately, at 
the expense of some repetitions. 
\paragraph{Acknowledgments}
It is a pleasure to thank \'Eric Gaudron and Ga\"el R\'emond for their efficiency in proving isogeny bounds which were even better than what 
they had promised, and for sharing their results with us. We are also grateful to the plafrim team in Bordeaux, who allowed us to make extensive 
computations on their machines, although what we eventually needed was less than we first feared.   
 
While working on this article we learnt that Fumiyuki Momose had passed away, in April of 2010. His work has been a great source of inspiration 
for us, and we would like to dedicate this article to his memory. 
   
\paragraph{Convention}
In this article we use the $O_1(\cdot)$-notation, which is a ``quantitative version'' of the familiar $O(\cdot)$-notation: ${A=O_1(B)}$ means ${|A|\le B}$.  

\section{Explicit bounds for integral points} 
\label{Runge}
Recall that, to a positive integer~$N$ and a 
subgroup~$G$ of $\GL_2(\Z/N\Z)$, one associates a modular curve of level (dividing)~$N$, denoted by~$X_G$. In particular, when ${N=p}$ is a prime number, 
and~$G$ is the normalizer of a split Cartan subgroup of $\GL_2(\F_p)$ (for instance, the subgroup of diagonal and anti-diagonal elements), the corresponding curve 
will be denoted by $X_\spl (p)$; it parametrizes geometric isomorphism classes  of elliptic curves endowed with an unordered pair of independent $p$-isogenies. For
$X_G$ any modular curve, we denote in the same way the Deligne-Rapoport model over~$\Z$, and by~$Y_G$ the scheme deprived of the cusps.   

In this section we prove the following explicit version of Theorem~1.1 from~\cite{BP11} (see Subsection~\ref{Sth1}). 
\begin{theorem}
\label{th1} 
For any prime number ${p\ge 3}$ and any ${P\in Y_{\mathrm{split}}(p)(\Z)}$ we have 
\begin{equation}
\label{erunge}
\height (P)=\height (j_P )\le 2\pi p^{1/2}+6\log p +{21(\log p)^2}{p^{-1/2}}.
\end{equation}
\end{theorem}
Here constants $2\pi$ and~$6$ are best possible for the method, but~$21$ can be refined, and can be replaced by~$3$ for sufficiently large~$p$.
The~$\Q$-isomorphism ${X_\spl(p)\simeq X_0^+(p^2)}$ shows that Theorem~\ref{th1} allows to tackle the case~$r=2$ in Theorem~\ref{MT}. To deal with the 
case~$r=3$, we will further need a fully explicit version of Theorem~7.3 from~\cite{BP10} about integral points on $X_0(p^r)$, $r\geq 2$ (subsection~\ref{ssx0pr}). 
By the Faltings height $\height_\calF(P)$ of a non-cuspidal point~$P$ on the curve $X_0(p^r)$ (or any modular curve) we mean the semi-stable Faltings 
height~$\height_\calF (E)$ of the underlying elliptic curve~$E$ (see~\cite{GR11}, section 2.3, for a discussion on different normalization choices; our~$\height_\calF$
is the~$\height_F$ of loc. cit.). 
\begin{theorem}
\label{tx0}
Let ${p\ge 3}$ be a prime number,~$K$ a quadratic number field with ring of integers~$\OO_K$, 
${r>1}$ an integer, and~$P$ a point of $Y_0 (p^r )(\OO_K )$. Then ${\height_\calF (P)\le 2p\log p+4p}$.
\end{theorem}
We follow the arguments of \cite{BP11} and \cite{BP10}, making explicit all the implicit constants occurring therein.  We shall routinely use 
the inequality\footnote{ We choose the principal determination of the logarithm, that is, for~${z \in \C}$ satisfying ${|z|<1}$, we set ${\log(1+z):=
-\sum_{k=1}^\infty (-z)^k/k}$.}
\begin{equation}
\label{eloglog}
\bigl|\log(1+z)\bigr| \le -\frac{\log(1-r)}r |z| \qquad \text{for $|z|\le r<1$}.
\end{equation}

\subsection{Siegel Functions}
\label{sest}

We denote by~$\HH$ the Poincar\'e upper half-plane and put ${\bar\HH=\HH\cup\Q\cup\{i\infty\}}$. 
For ${\tau\in \HH}$ we, as usual,  put ${q=q(\tau)=e^{2\pi i\tau}}$. For a rational number~$a$ we define ${q^a=e^{2\pi i a\tau}}$. 
Let ${\bfa=(a_1,a_2)\in \Q^2}$ be such that ${\bfa\notin \Z^2}$, and let ${g_\bfa:\HH\to \C}$ be the corresponding \textsl{Siegel function} 
\cite[Section~2.1]{KL81}. Then we have the following infinite product presentation for~$g_\bfa$ \cite[page~29]{KL81}:
\begin{equation}
\label{epga}
g_\bfa(\tau)= -q^{B_2(a_1)/2}e^{\pi ia_2(a_1-1)}\prod_{n=0}^\infty\left(1-q^{n+a_1}e^{2\pi ia_2}\right)\left(1-q^{n+1-a_1}e^{-2\pi i a_2}\right),
\end{equation}
where ${B_2(T)=T^2-T+1/6}$ is the second Bernoulli polynomial. 

The following is a quantitative version of (slightly modified) Proposition~2.1 from~\cite{BP11}. Let~$D$ be the familiar fundamental 
domain of $\SL_2(\Z)$ (that is, the hyperbolic triangle with vertices $e^{\pi i/3}$, 
$e^{2\pi i/3}$ and $i\infty$, together with the geodesic segments ${[i,e^{2\pi i/3}]}$ 
and ${[e^{2\pi i/3},i\infty]}$) and ${D+\Z}$ the union of all translates of~$D$ by the rational integers.
\begin{proposition}
\label{pga}
Assume that ${0\le a_1<1}$. Then for ${\tau\in D+\Z}$  we have
$$
\log \left|g_\bfa(\tau)\right|=\frac12B_2(a_1)\log|q|+ \log\left|1-q^{a_1}e^{2\pi ia_2}\right|+
\log\left|1-q^{1-a_1}e^{-2\pi ia_2}\right|+O_1(3|q|).
$$
\end{proposition}

\begin{proof}
We only have to show that
$$
\left|\sum_{n=1}^\infty\left(\log \left|1-q^{n+a_1}e^{2i\pi a_2}\right| +\log\left|1-q^{n+1-a_1}e^{-2i\pi a_2}\right|\right)\right|\le 3|q|.
$$
But this is inequality~(11) from~\cite{BP10}. We may notice that in~\cite{BP10} it is assumed that ${\tau \in D}$, but what is actually used is the 
inequality~${|q(\tau)|\le e^{-\pi\sqrt3}}$, which holds for every ${\tau\in D+\Z}$. \qed
\end{proof}

\subsection{A Modular Unit}
\label{ssmod}

In this subsection we briefly recall the ``modular unit'' construction. See \cite[Section~3]{BP11} for more details. 

Let~$N$ be a positive integer. Then for ${\bfa,\bfa'\in (N^{-1}\Z)^2\smallsetminus\Z^2}$ such that~$\bfa \equiv \bfa' \mod \Z^2$, we have~${g_\bfa^{12N}=
g_{\bfa'}^{12N}}$.  Hence the function~$g_\bfa^{12N}$ is well-defined for~$\bfa$ in ${\left(N^{-1}\Z/\Z\right)^2\smallsetminus\{0\}}$.  
The function ${u_\bfa=g_\bfa^{12N}}$ is $\Gamma(N)$-automorphic and hence defines a rational function on the modular curve $X(N)(\C )$; in fact, it 
belongs to the field $\Q(\zeta_N)\bigl(X(N)\bigr)$.

Now assume that ${N=p\ge 3}$ is an odd prime number,  and denote by $p^{-1}\F_p^\times$ the set of non-zero elements of ${p^{-1}\Z/\Z}$. Put 
$$
A= \bigl\{(a,0) : a\in p^{-1}\F_p^\times\bigr\}\cup \bigl\{(0,a) : a\in p^{-1}\F_p^\times\bigr\}, \qquad
U=\prod_{\bfa\in A}u_\bfa.
$$
Then~$U$ is $\Gamma_\spl(p)$-automorphic; in particular, it defines a rational function on $X_\spl (p)$, also denoted by~$U$; in fact, ${U\in \Q(X_\spl (p))}$.  

More generally, for ${c\in \Z}$ put 
$$
\beta_c=\begin{pmatrix}1&0\\c&1\end{pmatrix}, \qquad U_c =U\circ \beta_c= \prod_{\bfa\in A\beta_c}u_\bfa
$$
(recall that~$u_\bfa \circ \gamma =u_{\bfa \gamma}$), so that ${U=U_0}$.  (Warning: for~$c$ non-divisible by~$p$ the function~$U_c$ is not 
$\Gamma_\spl (p)$-automorphic!) The following is a quantitative version of Proposition~3.3 from~\cite{BP11}.

\begin{proposition}
\label{pu}
For 
${\tau\in D+\Z}$   we have 
$$
\log \left|U_c(\tau)\right| 
= \left\{
\begin{aligned} 
&(p-1)^2\log |q|+O_1\left( 4\pi^2 \frac{p^2}{\log |q^{-1}|} + 12p\log p +77p^2|q|\right) & {\mathrm{if}}\ p\mid c,\\
&-2(p-1)\log |q|+O_1\left( 8\pi^2 \frac{p^2}{\log |q^{-1}|} +72p^2|q|\right) &{\mathrm{if}}\ p\nmid c,
\end{aligned}
\right.
$$
where we write ${q=q(\tau)}$. 
\end{proposition}

For the proof of Proposition~\ref{pu}  we need a slight sharpening of Lemma~3.5 from~\cite{BP11}.

\begin{lemma}
\label{llogz}
Let~$z$ be a complex number, ${|z|<1}$, and~$N$ a positive integer. Then  
\begin{equation}
\label{eeta}
\left|\sum_{k=1}^N\log\bigl|1-z^k\bigr| \right|\le \frac{\pi^2}6\frac1{\log |z^{-1}|}.
\end{equation}
\end{lemma}

\begin{proof}
We have 
${\bigl|\log|1+z|\bigr|\le -\log(1-|z|)}$
for ${|z|<1}$. 
Hence it suffices to prove the inequality
\begin{equation}
\label{eqq}
-\sum_{k=1}^\infty\log(1-q^k) \le \frac{\pi^2}6\frac1{\log (q^{-1})}  \qquad ({\mathrm{for}}\ 0<q<1).
\end{equation} 
Using~\eqref{eloglog} with~$q$ instead of~$z$ and with ${r=1/2}$, we 
 find that for ${0<q\le 1/2}$
$$
-\sum_{k=1}^\infty\log(1-q^k) \le (4\log 2)  q \le \frac{4\log 2}e \frac1{\log (q^{-1})}  < \frac{\pi^2}6\frac1{\log (q^{-1})},
$$
which proves~\eqref{eqq} for ${0<q\le 1/2}$. We are left with ${1/2\le q<1}$. 

Put ${\tau = \log q/(2\pi i)}$. Then 
$$
-\sum_{k=1}^\infty\log(1-q^k) = \frac1{24}\log q -\log|\eta(\tau)|,
$$
where $\eta(\tau)$ is the Dedekind $\eta$-function. Since ${|\eta(\tau)|=|\tau|^{-1/2}|\eta(-\tau^{-1})|}$, we have 
\begin{equation}
\label{eqeta}
-\sum_{k=1}^\infty\log(1-q^k) = -\frac1{24}\log Q+\frac1{24}\log q+\frac12\log|\tau| -\sum_{k=1}^\infty\log(1-Q^k)
\end{equation}
with ${Q=e^{-2\pi i\tau^{-1}}= e^{4\pi^2/\log q}}$. The first term on the right of~(\ref{eqeta}) is exactly ${(\pi^2/6)/\log(q^{-1})}$, and the 
second term is negative for ${0<q<1}$. To complete the proof, we must show that, when ${1/2\le q<1}$, the sum of the remaining two terms is 
negative. 

Indeed, when ${1/2\le q<1}$, we have 
$$
\frac12\log|\tau| \le -\frac12\log \frac{2\pi}{\log 2} \le -1, \qquad Q \le e^{-4\pi^2/\log 2} \le 10^{-24}.
$$
Applying~\eqref{eloglog} with~$Q$ instead of~$z$ and with ${r=10^{-24}}$, we bound the fourth term in~\eqref{eqeta} by ${10^{-23}}$. Hence the 
sum of the third and the fourth terms is negative, as wanted.\qed
\end{proof}

\paragraph{Proof of Proposition~\ref{pu}}
For ${a\in \Q/\Z}$ we denote by~$\tilde a$ the lifting of~$a$ to the interval $[0,1)$. 
Then 
for ${\tau\in D+\Z}$  we deduce from Proposition~\ref{pga} that 
\begin{equation}
\label{ewa}
\log \left|U_c(\tau)\right|=6p\Sigma_1\log |q| +12p\Sigma_2+ O_1(72p^2|q|),  
\end{equation}
where
$$
\Sigma_1=\sum_{\bfa\in A\beta_c}B_2(\tilde a_1),\qquad
\Sigma_2=\sum_{\bfa\in A\beta_c}\Bigl( \log\bigl|1-q^{\tilde a_1}e^{2\pi ia_2}\bigr|+ \log\bigl|1-q^{1-\tilde a_1}e^{-2\pi ia_2}\bigr|\Bigr).
$$
Now we are going to calculate~$\Sigma_1$, using the identity
$$
\sum_{k=1}^{N-1}B_2\left(\frac kN\right) =-\frac{(N-1)}{6N},
$$ 
and to estimate $\Sigma_2$ using  Lemma~\ref{llogz}.

If  ${p\mid c}$ then~${A\beta_c=A}$ and 
\begin{align}
\label{ebern}
\Sigma_1&= \sum_{k=1}^{p-1}B_2\left(\frac kp\right)+ (p-1)B_2(0)= \frac{(p-1)^2}{6p},\\
\label{esumlo}
\Sigma_2&= 2\sum_{k=1}^{p-1}\log\bigl|1-q^{k/p}\bigr| +\log\left|\frac{1-q^p}{1-q}\right|+\log p.
\end{align}
Lemma~\ref{llogz} with ${z=q^{1/p}}$ implies that 
$$
\Bigl| \sum_{k=1}^{p-1}\log |1-q^{k/p} | \Bigr| \le \frac{\pi^2}6\frac p{\log |q^{-1}|}.
$$
 Also, since ${|q|\le e^{-\pi\sqrt3}}$, we have ${\Bigl| \log\left|1-q\right|  \Bigl| \le 1.01|q|}$ and
${\Bigl| \log\left|1-q^p\right|  \Bigl| \le 1.01 |q|^p\le 0.01|q|}$.
Combining all this with~(\ref{ewa}),~(\ref{ebern}) and~(\ref{esumlo}), we  prove the proposition in the case ${p\mid c}$.

If ${p\nmid c}$
then
${A\beta_c=\{(a,0) : a\in p^{-1}\F_p^\times\}\cup\{(a,ab) : a\in p^{-1}\F_p^\times\}}$, 
where  ${bc\equiv 1\bmod p}$. Hence 
\begin{align*}
\Sigma_1=& 2\sum_{k=1}^{p-1}B_2\left(\frac kp\right)=-\frac{p-1}{3p}, \\
\Sigma_2=& 
2\sum_{k=1}^{p-1}\log\bigl|1-q^{k/p}\bigr|+ 2\sum_{k=1}^{p-1}\log\bigl|1-(q^{1/p}e^{2\pi ib/p})^k\bigr|.
\end{align*}
Using Lemma~\ref{llogz} with ${z=q^{1/p}}$ and with ${z= q^{1/p}e^{2\pi ib/p}}$, we  complete the proof.   \qed

\subsection{Proof of Theorem~\ref{th1}}
\label{Sth1}
We set~$G$ as the subgroup of diagonal and anti-diagonal matrices in~$\GL_2 (\F_p)$ and choose the corresponding modular
curve as a model for~$X_\spl (p)$. Define the ``modular units''~$U_c$ as in Subsection~\ref{ssmod}.  Recall that ${U=U_0}$ 
belongs to the field $\Q(X_\spl (p))$. Theorem~\ref{th1} is a consequence of Proposition~\ref{pu} and the following statement, which 
is Proposition~4.2 from~\cite{BP11}.

\begin{proposition} 
\label{pari}
For  ${P\in Y_\spl (p)(\Z)}$ we have ${0\le \log|U(P)| \le 24 p\log p}$. \qed
\end{proposition}

We are ready now to prove Theorem~\ref{th1}. Let~$p\ge 3$ and~${P\in Y_\spl (p)(\Z)}$. According to Lemma~3.2 from~\cite{BP11}, 
there exists ${\tau\in D+\Z}$ and ${c\in \Z}$ with ${U_c(\tau)=U(P)}$ and ${j(\tau)=j(P)}$.  We write ${q=q(\tau)}$. Recall that~$j(\tau )$ 
and~$q(\tau )$ are real numbers, and that $\height (j(\tau ))=\log \vert j(\tau )\vert$ if~$j(\tau )\in \Z$. It suffices to show that 
\begin{equation}
\label{eqmione}
\log|q^{-1}|\le 2\pi p^{1/2}+ 6\log p + {20 (\log p)^2}{p^{-1/2}}.
\end{equation}
Indeed, we may assume that ${|j(\tau)|\ge 3500}$ (otherwise~\eqref{erunge} holds trivially), in which case Corollary~2.2 of~\cite{BP10}  
gives ${\bigl|j(\tau)-q^{-1}\bigr|\le 1100}$. Hence, using the inequality
$$
\log |a|\le \log |b|+ \frac{|a-b|}{|a|-|a-b|},
$$
which holds for real numbers~$a$ and~$b$ with same sign (and $0<|b|<|a|$ or $0<|a|<|b|<|2a|$), we obtain
$$
\log|j(\tau)|\le \log|q^{-1}|+ \frac{1100}{|j(\tau)|-1100}.
$$
Now using~\eqref{eqmione} and assuming that ${\log|j(\tau)|\ge 2\pi p^{1/2}+ 6\log p}$, we obtain 
$$
\log|j(\tau)|\le 2\pi p^{1/2}+ 6\log p + 20\frac{(\log p)^2}{p^{1/2}} + \frac{1100}{p^6e^{2\pi p^{1/2}}-1100} \le 2\pi p^{1/2}+ 6\log p + 
21\frac{(\log p)^2}{p^{1/2}},
$$
as wanted.

Let us prove~\eqref{eqmione}. 
Assume first that
 ${p\nmid c}$. Using Propositions~\ref{pu} and~\ref{pari} and assuming that ${\log|q^{-1}|\ge 2\pi p^{1/2}+6\log p}$, we obtain
\begin{align*}
\log|q^{-1}| &\le \frac{\log |U_c(\tau)|}{2(p-1)} + \frac{4\pi^2p^2}{p-1}\frac1{\log |q^{-1}|} +36\frac{p^2}{p-1}|q|\\
&\le\frac{12p\log p}{p-1}+ \frac{4\pi^2p}{\log |q^{-1}|} + \frac{4\pi^2p}{p-1}\frac1{\log |q^{-1}|}+54p|q|\\
&\le 12\log p+ \frac{12\log p}{p-1}+\frac{4\pi^2p}{\log |q^{-1}|}+ \frac{2\pi p^{1/2}}{p-1}+ 54p^{-5}e^{-2\pi p^{1/2}}\\
&\le 12\log p + \frac{4\pi^2p}{\log |q^{-1}|}+\frac{21}{p^{1/2}}.
\end{align*}
It follows that ${\log|q^{-1}|}$ does not exceed the largest root of the quadratic polynomial
$$
f(T)=T^2- \left(12\log p +{21}{p^{-1/2}}\right)T- 4\pi^2 p,
$$
that is,
\begin{align}
\log|q^{-1}|&\le \left(4\pi^2 p+ \left(6\log p +{10.5}{p^{-1/2}}\right)^2\right)^{1/2}+6\log p+ 10.5p^{-1/2}\nonumber\\
\label{emidd}
&\le 2\pi p^{1/2} + \frac{\left(6\log p +{10.5}{p^{-1/2}}\right)^2}{4\pi p^{1/2}}+6\log p+ 10.5p^{-1/2}\\
&\le 2\pi p^{1/2}+6\log p + {20(\log p)^2}{p^{-1/2}}, \nonumber
\end{align}
where we use the inequality ${(a+b)^{1/2}\le a^{1/2}+(1/2) ba^{-1/2}}$ in~\eqref{emidd}. This completes the proof 
of~\eqref{eqmione} in the case ${p\nmid c}$.

In the case ${p\mid c}$ Proposition~\ref{pu} gives
$$
\log|q^{-1}| \le -\frac{\log |U_c(\tau)|}{(p-1)^2} + \frac{4\pi^2 p^2}{(p-1)^2}\frac1{\log|q^{-1}|}+ \frac{12p\log p}{(p-1)^2} + \frac{77 p^2}{(p-1)^2}|q|.
$$
Proposition~\ref{pari} implies that  ${-\log |U_c(\tau)|\le 0}$. Assuming that ${\log|q^{-1}|\ge 2\pi p^{1/2}+6\log p}$, we obtain
$$
\log|q^{-1}| \le  \frac{2\pi p^{3/2}}{(p-1)^2}+ \frac{12p\log p}{(p-1)^2} + \frac{77}{(p-1)^2p^4}e^{-2\pi p^{1/2}}\le 19,
$$
which is sharper than~\eqref{eqmione}. The theorem is proved.  
\qed
\subsection{Integral points on $X_0(p^r)$: proof of Theorem~\ref{tx0}}
\label{ssx0pr}
Let ${p}$ be a prime number as usual. 
We will use the following double-covering of $X_\spl (p)$. Let denote by $X_\splic(p)$ the curve  corresponding to a split Cartan subgroup of $\GL_2(\F_p)$ ({\it not} its 
normalizer),  for instance the diagonal subgroup (see  the beginning of  Section~\ref{Runge}). 
It parametrizes geometric isomorphism 
classes  of elliptic curves endowed with an \textsl{ordered} pair of independent $p$-isogenies. Factorizing by the natural involution that switches the isogenies (which 
is induced by the matrix $\left(\begin{smallmatrix}0&1\\-1&0\end{smallmatrix}\right)$ acting on the Poincar\'e half-plane~$\HH$) defines a degree-2 covering ${X_\splic(p)
\to X_\spl(p)}$. On the other hand, there is an isomorphism ${\phi: X_{0} (p^2 )\to X_\splic(p)}$ over~$\Q$ defined functorially as
\begin{equation}
\label{ex0spc}
\bigl(E,A\bigr)\mapsto \bigl(E/B , (B^\ast,C)\bigr),
\end{equation}
where ${A=C\circ B}$ is the obvious decomposition of the cyclic $p^2$-isogeny~$A$ into the product of two $p$-isogenies and~$B^*$ is the dual isogeny. On the 
Poincar\'e upper half-plane~$\HH$, the map~$\phi$ is induced by~${\tau \mapsto p\tau}$. 

This interplay between the isomorphic  curves might look a bit confusing 
at first sight, but each point of view has its own advantages. In particular, replacing~$X_{0} (p^2 )$ by $X_\splic(p)$ (that is, a level~$p^2$-structure by a~$p$-structure) 
is significantly more advantageous for  Runge's method.

Furthermore, curves $X_0^+(p^2)$ and $X_\spl(p)$ are quotients of $X_0(p^2)$ and $X_\splic(p)$, respectively, by natural involutions, and a straightforward verification shows that~\eqref{ex0spc} defines a $\Q$-isomorphism ${X_{0}^+ (p^2 )\to X_\spl(p)}$.

We deduce Theorem~\ref{tx0} from the following result, which is  Theorem~6.1 from~\cite{BP10}. 

\begin{theorem}
\label{tspc}
Let ${p\ge 3}$ be a prime number and~$K$ a number field of degree at most~$2$. Then for 
a point ${P\in Y_\splic(p)(\OO_K)}$ we have ${\height(P)\le 24p\log (3p)}$. 
\end{theorem}

We shall need some basic estimates concerning the Faltings height.
 
\begin{proposition} 
\label{pfal}
\begin{enumerate}
\item
\label{ifalt}
Let~$E$ and~$E'$ be isogenous elliptic curves over some number field, connected by an isogeny of degree~$\delta$. Then
${\left|\height_\calF(E)-\height_\calF(E')\right|\le (1/2)\log\delta}$. 

\item
\label{ifalj}
For an elliptic curve~$E$ we have ${\height_\calF(E)\le (1/12)\height(j_E)+3}$. 
\end{enumerate}
\end{proposition}
Item~(\ref{ifalt}) is a well-known result of Faltings \cite[Lemma 5]{Fa83}. Item~(\ref{ifalj}) is, basically, due to Silverman \cite[Proposition~2.1]{Si84}, who proved the 
inequality  ${\height_\calF(E)\le (1/12)\height(j_E)+C}$ with an unspecified absolute constant~$C$. The calculations of Pellarin on pages 240--241 of~\cite{Pe01} 
imply that ${C=4}$ would do, though he does not state this explicitly.  It finally follows from Gaudron and R\'emond \cite[Lemma~7.9]{GR11} 
that ${C=3}$ would do.

\paragraph{Proof of Theorem~\ref{tx0}}
We may assume ${r=2}$. Let ${\phi:X_0(p^2)\to X_\splic(p)}$ be the isomorphism 
defined by~\eqref{ex0spc}. Then the elliptic curve implied by a
point~$P$ on $X_0(p^2)$ is $p$-isogenous to the curve implied by the point ${P'=\phi(P)}$ 
on $X_\splic(p)$. Proposition~\ref{pfal} implies that 
$$
\height_{\calF} (P) \le
\height_{\calF} (P') + \frac{1}{2}\log p, \qquad \height_{\calF}(P')\le \frac1{12}\height(P')+3.
$$ 
Finally, Theorem~\ref{tspc} applied to the point~$P'$ gives
${\height(P')\le 24p\log (3p)}$. Combining all this, we obtain
$$
\height_\calF(P)\le 2p\log (3p)+\frac12\log p+3\le 2p\log p+4p,
$$
as wanted. \qed

\section{An Upper Bound for~$p$}
\label{Isogeny}
The main result of this section is Theorem~\ref{tse}. It is an explicit version of Theorem 1.3 from~\cite{BP10}, which covers Theorem~1.2 from~\cite{BP11}. Our 
previous work relied on Pellarin's refinement~\cite{Pe01} of Masser-W\"ustholz famous upper bound~\cite{MW90} for the smallest degree of an isogeny between 
two elliptic curves. Here we invoke the very recent improvement on Pellarin's bound, due to Gaudron and R\'emond \cite[Theorem~1.4]{GR11}, with much 
sharper numerical constants. 
\begin{theorem}[Gaudron and R\'emond]
\label{tgr}
Let~$E$ be an elliptic curve 
defined over a number field~$K$ of degree~$d$. Let~$E'$ be another elliptic curve, defined over~$K$ and isogenous
to~$E$ over~$\bar K$. Then there exists an isogeny  ${\psi:E\to E'}$ of degree at most
${10^7 d^2 \bigl(\max \{ \height_\calF (E),  985\}+4\log d \bigr)^2}$. 
\end{theorem}
We combine Theorems~\ref{th1}, \ref{tx0} and~\ref{tgr} to prove the following.  
\begin{theorem}
\label{tse}
\begin{enumerate}
\item 
\label{i2}
For ${p>1.4\cdot 10^7}$, every point in~$X_0^+ (p^2 )(\Q)$ is either a CM point or a cusp.
\item 
\label{i3}
For $p>1.7\cdot 10^{11}$, every point in~$X_0^+ (p^3 )(\Q )$ is either a CM point or a cusp.
\end{enumerate}
\end{theorem}

A numerically sharper version of item~(\ref{i2})  is also given 
in~\cite{GR11}. Our version is sufficient for our purposes. 

We shall use Theorem~\ref{tgr} through its following immediate consequence.
\begin{proposition}
\label{ppel}
Let~$E$ be a non-CM elliptic curve  defined over a number field~$K$ of degree~$d$, and 
admitting a cyclic isogeny over~$K$ of degree $\delta$. Then ${\delta\le 10^7 d^2 \bigl(\max \{ \height_\calF (E),  985\}+4\log d \bigr)^2}$.
\end{proposition}
\begin{proof} 
Let $\phi$ be a cyclic isogeny from~$E$ to~$E'$. Let ${\psi \colon E\to E'}$ be an isogeny of degree bounded by 
${10^7 d^2 \max \left( \height_\calF (E) +4\log d , 10^3 \right)^2}$ granted by Theorem~\ref{tgr}, and let ${\psi^* \colon E' \to E}$ be
the dual isogeny.  As~$E$ has no CM, the composed map ${\psi^* \circ \phi}$ must be 
multiplication by some integer $n$ and then $n^2=\mathrm{deg}(\phi)\mathrm{deg}(\psi).$ Since $\phi$ is cyclic,  $\mathrm{deg}(\phi)\leq |n|$. It follows  that $\mathrm{deg}(\phi)\leq\mathrm{deg}(\psi)$ or $\mathrm{deg}(\phi)= |n| $ and ${\phi =\pm \psi}$.  
\qed
\end{proof}
\paragraph{Proof of Theorem~\ref{tse}}

We start with item~(\ref{i2}). Let~$Q$ be a non-cuspidal and non-CM point in~$X_0^+ (p^2 )(\Q )$, and let~$P$ be the corresponding point in 
$X_\spl(p)(\Q)$ defined by~\eqref{ex0spc}. Let~$E_1$ and~$E_2$ be the elliptic curves corresponding to~$Q$ (defined over a quadratic extension of~$\Q$) and 
let~$E$ be the elliptic curve associated with~$P$.

Since~$E$ and~$E_1$ are $p$-isogenous, Proposition~\ref{pfal} implies that
\begin{equation}
\label{eeonee}
\height_\calF (E_1)\le \height_\calF (E)+\frac12\log p 
\le \frac1{12}\height(j_E) +\frac12\log p+3 .
\end{equation}

A result of Mazur, Momose and Merel (see Theorem~6.1 in~\cite{BP11}) implies that  ${j(P)=j_E\in \Z}$; in particular, 
${\height(j_E)=\log|j_E|}$. Hence we may use Theorem~\ref{th1}, which yields 
\begin{equation}
\label{ee}
\frac1{12}\height(j_E)\le \frac{2\pi}{12}p^{1/2} + \frac12\log p+\frac{21}{12} \frac{(\log p)^2}{p^{1/2}}.
\end{equation}

{\sloppy

On the other hand, since the curve~$E_1$ admits a cyclic $p^2$-isogeny over a quadratic field, Proposition~\ref{ppel} implies that 
${p^2\le 4\cdot 10^7  \bigl(\max \{ \height_\calF (E_1),  985\}+4\log 2 \bigr)^2}$.
It follows  that 
${p\le 7\cdot 10^3 \max \{ \height_\calF (E_1),  
985\}}$, 
that is, either ${p\le 7\cdot 10^6}$ and we are done, or  ${p\le 7\cdot 10^3  \height_\calF (E_1)}$. In this latter case, using~\eqref{eeonee} and~\eqref{ee}, we obtain
\begin{equation}
\label{eineq}
p\le 7\cdot 10^3  \left(\frac{2\pi}{12}p^{1/2} +\log p+3+\frac{21}{12} \frac{(\log p)^2}{p^{1/2}}\right).
\end{equation}
One readily checks that for ${p\ge 10^7}$ the right-hand side of~\eqref{eineq} does not exceed ${3.71\cdot 10^3p^{1/2}}$, 
which implies that ${p\le 1.4\cdot 10^7}$. This proves item~(\ref{i2}).

For the proof of item~(\ref{i3}) we play the same game, in a more straightforward way. Let~$Q$ be a non-CM non-cuspidal point on $X_0^+ (p^3 )(\Q)$.
Let~$Q_1$ be one of its lifts in $Y_0 (p^3 )(K)$, where~$K$ is a quadratic field,  and let~$E$ be the underlying elliptic curve. 
By Theorem~8.1 of~\cite{BP10} we still know that $j(Q_1 )=j_E$ belongs to $\OO_K$. The curve~$E$ is endowed with a cyclic isogeny of degree~$p^3$ over~$K$. 
Proposition~\ref{ppel} gives ${p^{3/2} \le 7\cdot 10^3 \max \{ \height_\calF (E),  985\}}$. So now either 
${p \le (7\cdot 10^6 )^{2/3} <4\cdot 10^4}$ 
and we are done, 
or  ${p^{3/2} \le 7\cdot 10^3  \height_\calF (E)}$. In the latter case Theorem~\ref{tx0} implies that ${p^{1/2} \le 7\cdot 10^3 (2\log p +4)}$, which can be re-written 
as ${ep^{1/2}\le 2.8\cdot 10^4 e\log (ep^{1/2})}$ (where ${e=2.718\dots}$). Since ${x/\log x \ge 2.8\cdot 10^4 e}$ for ${x\ge 1.1\cdot10^6}$, we obtain ${ep^{1/2}< 
1.1\cdot10^6}$, which implies ${p<1.7\cdot 10^{11}}$, as wanted. 
\qed

}

\section{The Heegner-Gross sieve}
\addtocontents{toc}{\vspace{-0.7\baselineskip}}
\label{Gross}
\subsection{Reminder on  Mazur's techniques and  Heegner-Gross vectors}
For the convenience of the reader, we here recall the strategy explained in \cite{Pa05}, paragraph 6, improved by the use of generalized jacobians
as in the work of Merel (\cite{Me05}). Those results are used in our algorithm. We refer to \cite{Pa05}, \cite{Re08} and \cite{Me05} for details. In all what
follows, we assume~${p\geq 11},\ {p\neq 13}$.

\subsubsection{Variant of Mazur's techniques}
Let $r>1$ an integer and $P$ be a non-cuspidal  and non-CM rational point on $\xrp$. The point $P$ gives rise to a point $x\in Y_0(p^r)(K)$ defined over a number  
field $K$ with $[K:\Q]\leq 2.$ By Mazur's results \cite{Ma77}, $K$ is  quadratic for $p\geq 11,p\neq 13.$ Let denote by $\pi_p:\xr\longrightarrow X_0(p)$ the natural 
morphism which preserves the~$j$-invariant. It is easy to see that if the points $x_1=\pi_p\circ w_{p^r}(x)$ and $x_2=w_{p}\circ \pi_p(x)$ are equal in 
$X_0(p)(K)$, then $x$ is a CM point which yields a contradiction. 
To study when this equality occurs,  we use a variant of techniques developped by Mazur in \cite{Ma78}.

Denote by $X_0(p)_\Z$ the normalization of $\mathbb P^1$ in $X_0(p) $ via $j:X_0(p)\longrightarrow X_0(1)\simeq \mathbb P^1$ and  by $Y_0(p)_\Z$ the open affine subscheme obtained by deleting the cusps. Recall $\ok$ denotes the ring of integers of $K$ and let $X_0(p)_\ok^{\mathrm{sm}}$ be the smooth part of $X_0(p)_\ok=
X_0(p)_\Z\times_\Z  \mathrm{Spec}(\ok)$ obtained by removing the supersingular points in characteristic $p$. Let $s_1,s_2 : \mathrm{Spec}(\ok)\longrightarrow 
X_0(p)_\ok$ the sections defined by $x_1,x_2$, respectively. The next Proposition follows from the work of Momose (\cite{Mo86}) and from~\cite{Pa05}.

\begin{proposition}\label{momose}
\begin{enumerate}
\item In the fibers of characteristic $p$, the sections $s_1$ and $s_2$ are not supersingular points and coincide ;
\item the field $K$ is a quadratic extension of~$\Q$ in which $p$ splits. 
\end{enumerate}
\end{proposition}

In the sequel, we adopt  the notations of \cite{Me05}:  we denote by $\jg$ the generalized jacobian of $X_0(p)$ with respect to the set of cusps  and by $\je$ the \emph{winding quotient} of $\jg$. Let  $\jg_\ok$ and $\je_\ok$ the respective N\'eron models over $\mathrm{Spec}(\ok)$. We consider the composition $\phi_P : Y_0(p)\longrightarrow \je$ of the canonical morphism $\jg\longrightarrow \je$ with the Albanese morphism $Y_0(p) \longrightarrow \jg$ which to a point $Q$ associates the class of  the divisor $[(Q)-(x_1)]$.
 By Proposition~\ref{momose},  one can extend $\phi_P$ to a morphism 
 $$
 \phi_P : Y_0(p)^{\mathrm{sm}}_\ok\longrightarrow \je_\ok
 $$  
 and the images $\phi_P(s_1)$ and $\phi_P(s_2)$ coincide in characteristic $p$.  Since any section of the identity component $\je_\ok^0$ of $\je_\ok$  is of finite order (see \cite{Me05}  Proposition~2), it follows that if $\phi_P$ is a formal immersion at ${s_1}_{\slash \F_p}$ then $s_1=s_2$ so $x_1=x_2$. We refer for instance to \cite{Me05}, 
 Proof of Proposition~6 in Section~4, for a detailed proof of this fact which is a variant of Mazur's techniques \cite{Ma78}. 
 
  Taking into account the particularity of the fibers in characteristic $p$ of $X_0(p)_\ok$, one can then give a criterion of formal immersion (\cite{Pa05}, \cite{Me05}). 
Let $\mathcal S$ be the finite set of isomorphism classes of supersingular elliptic curves in characteristic $p$. There is an isomorphism between $\mathrm{Cot}_0(\jg_{\F_p})$ and $\F_p^\mathcal S$. Both can be endowed with a structure of Hecke module compatible with this isomorphism. Any $v=\sum_{s\in\mathcal S}\lambda_s[s]\in \F_p^\mathcal S$ corresponds to an element $\omega_v$ of $\mathrm{Cot}_0(\je_{\F_p})$ if and only if $I_e^{\sharp}v=0$ where we denote by $I_e^{\sharp}$ the winding ideal of the Hecke algebra (see \cite{Me05}, proof of Proposition~4). Moreover, taking the modular function $j$ as a local parameter for $Y_0(p)_{\F_p}$ in the neighborhood of $s_1/\F_p$, we have $\mathrm{Cot}(\phi_P)(\omega_v )=\sum_{s\in\mathcal S}\frac{\lambda_s}{j(P)-j(s)}dj$.  It allows to prove  the following proposition (\cite{Pa05,Me05}, see also~\cite{Me01}) . 
\begin{proposition}\label{if}
Let $s_1 \in Y_0(p)^{\mathrm{sm}}_{\Z_p}(\Z_p)$ be a section, $P$ the point obtained by restriction to the generic fiber and $j(P)$ his $j$-invariant. 
Suppose that  there exists $v=\sum_{s\in\mathcal S}\lambda_s[s]\in\Z^{\mathcal S}$ such that $I_e^{\sharp}v=0$ and  $\sum_{s\in\mathcal S}\frac{\lambda_s}{j(P)-j(s)}\neq 0$ in $\F_{p^2}$,   then $\phi_P$ is a formal immersion at ${s}_{\slash\F_p}.$
\end{proposition}

With the variant of Mazur's techniques explained above, this gives the corollary (see \cite{Me05} Proposition~6 for this formulation):
\begin{corollary}[\cite{Pa05, Me05}]\label{criterion}
  If for all  ordinary invariant $j_0\in\F_p$, there exists $v=\sum_{s\in\mathcal S}\lambda_s[s]\in\Z^{\mathcal S}$ such that $I_e^{\sharp}v=0$ and  $\sum_{s\in\mathcal S}\frac{\lambda_s}{j_0-j(s)}\neq 0$ in $\F_{p^2}$,   then $X_0^+(p^r)(\Q)$ is trivial for all $r>1.$ 
 \end{corollary}
  
\begin{remark}\label{genjac}
The use of generalized jacobians is not necessary (and was not made in~\cite{Pa05} nor in~\cite{Re08}), but it
allows to give a neater formulation to the criterion of Proposition~\ref{critfinal} below. As an illustration, one can check that under this new form it 
readily gives triviality of~$X_0^+ (37)(\Q )$, for instance, whereas the previous version could not deal with this case 
and we had to invoke instead peculiar studies of level 37 by Hibino, Murabayashi, Momose and Shimura (cf.~\cite{HM97}, 
\cite{MS02}), as discussed in Section~6, page~9 of~\cite{Pa05}. 
\end{remark}

\subsubsection{Heegner-Gross vectors}
In \cite{Pa05}, the second named author made use of a formula of Gross to exhibit some elements $e_D\in\Z^{\mathcal S}$ such that $I_e^{\sharp}e_D=0$. Let indeed~$-D$ 
be a  quadratic imaginary discriminant and $\mathcal O_{-D}$ the order of discriminant~$-D.$ Let $s\in \mathcal S$ be the isomorphism class of a supersingular elliptic curve $E_s$ in characteristic $p$. The ring $R_s=\mathrm{End}_{\F_{p^2}}(E_s)$  is a maximal order of the quaternion algebra $\mathcal B$ ramified at $p$ and $\infty$. Moreover, the elements of $\mathcal S$ are in one-to-one correspondance with the set of maximal orders of $\mathcal B.$ The quadratic field $L=\Q(\sqrt{-D})=\mathcal O_{-D}\otimes \Q$ embeds in $\mathcal B$ if and only if $p$ is ramified or inert in $L$ and we then denote by $h_s(-D)$ the number of optimal embeddings of $\mathcal O_{-D}$ in $R_s$ modulo conjugation by $R_s^\times$ (an embedding is optimal if it does not extend to any larger order).  
We now define
\begin{equation}e_D=\frac 1{|\mathcal O_{-D}^\times|}\sum_{s\in \mathcal S}h_s(-D)[s]
\end{equation} which we consider as an element of $\frac 1{12} \Z^\mathcal S.$ 

\begin{proposition}[\cite{Pa05, Me05}]\label{gross}
We have $I_e^{\sharp} e_D=0$. 
\end{proposition}

This is a slightly modified version of Proposition~4.1 of \cite{Pa05} as explained in \cite{Me05},    Proposition~5 and Corollary of Theorem~6  (see Remark~\ref{genjac}).

The $h_s(-D)$ optimal embeddings of $\mathcal O_{-D}$ in $R_s$ modulo conjugation by $R_s^\times$ are in one-to-one correspondence with the pairs $(E,f)$, where  $E$ is an elliptic curve with CM by $\mathcal O_{-D}$, which are isomorphic to $E_s$ in characteristic $p$ and $f$ is  a given  isomorphism $\mathcal O_{-D}\cong \mathrm{End}(E)$ (see for instance \cite{Gr87}).  So for  $p$ inert or ramified in $L$, the vector $e_D$ is  the sum of isomorphism classes of elliptic curves which are the  reduction  in characteristic $p$ of   elliptic curves having CM by $\mathcal O_{-D}$. The differential associated to $e_D$ is then just equal to the mod $p$ logarithmic derivative: 
$$
\frac{H_{-D}'(j)}{H_{-D}(j)}dj,
$$ 
where $H_{-D}=\prod_{E ; \mathrm{End}(E)\cong\mathcal O_{-D}}(X-j(E))$ is the Hilbert class polynomial associated 
with $-D$.  Applying this to Corollary~\ref{criterion} we obtain the following criterion (recall we always assume ${p\geq 11}$, ${ p\neq 13}$). 
\begin{proposition}\label{critfinal} If for all  ordinary invariant $j_0\in\F_p$, there exists a quadratic imaginary discriminant $-D<0$ such that $p$ is inert or 
ramified in $\Q(\sqrt{-D})$ and  $H_D'(j_0)\neq 0$  in $\F_p$, then $X_0^+(p^r)(\Q)$ is trivial for all integers ${r>1}$. 
 \end{proposition}

\subsection{The sieve}\label{explsieve}
We actually use even  a more restrictive criterion.

\begin{corollary}\label{sieve}
Let $-D$ be a fundamental quadratic imaginary discriminant and $\chi_D$ the associated
quadratic Dirichlet character. For a positive integer $c$, write  $R_{c,D} :={\mathrm{Res}} (H'_{-D} ,H'_{-c^2 D} )$ the integer resultant. Suppose that  $p>11,\  p\neq 13$ is a prime 
such that $\chi_D(p)=0$ or $-1$ and\footnote{We use the ``French'' notation $[\![a,b]\!]$ for the set of integers $x$ satisfying $a\le x\le b$.}  $p\nmid r_D:={\mathrm{gcd}}(R_{c,D}; c\in [\![2,7]\!])$.    Then $X_0^+(p^r)(\Q)$ is trivial for all integer $r>1.$
\end{corollary}

\begin{proof}
Let $p$ be a prime as in the proposition. Then there exists $c\in[\![2,7]\!]$ such that $R_{c,D}\neq 0\mod p$. (This range of conductors is of course only motivated 
by our computational needs.) So  for all ordinary $j_0\in \F_p$ either $H_{-D}'(j_0)$ or $H_{-c^2D}'(j_0)$ is non-zero. Moreover,   $p$ is inert or ramified in $\Q(\sqrt{-D})$. 
The result follows from Proposition~\ref{critfinal}. \qed
\end{proof} 

\medskip
We are now ready to state our algorithm.

\paragraph{ALGORITHM, part I:} 

 Fix a bound $N$ and a list $\cal D$ of  quadratic imaginary discriminants: in the sequel, we eventually take $N=10^{14}$ and choose the discriminants $-D$  of 
 class number $h(-D)\leq 4$ (and also $-D=-87$ which is of class number $6$) to obtain Hilbert class polynomials of  small degree $\mathrm{deg}(H_{-D})=h(-D)$. For each $-D\in \cal D$
 we compute the prime factors $\neq 13$ in $[\![11,N-1]\!]$ of $r_D$.  In this way, we construct step by step a list $\mathcal L$ of fundamental quadratic imaginary discriminants and  a list $\bad$ of prime numbers having the following property: 
\begin{center}$(\star)\quad$\emph{ if $p<N$ is a prime number such that  $p\not \in \bad $ and  $\chi_d(p)\in\{0,-1\}$ for some ${-d\in \mathcal L}$,   then $X_0^+(p^r)(\Q) $ is trivial for all $r>1.$}
\end{center}
We also construct a list $\good$ which is useful within the procedure (see below). 

\medskip
\noindent \textbf{Details:} 
\begin{enumerate}
\item If the class number is one, then $H_{-D}$ is of degree one and unitary so $H_{-D}'=1$. We initialize $\mathcal L$ to $\mathcal L=\{ -3,-4,-7,-8,-11,-19,-43,-67,-163\}$, 
and $\good$ and $\bad$ to the empty lists.
\item Let  $-D\in\mathcal D$  not yet in $\mathcal L$ and  $p \in[\![11,N]\!]$ a prime factor of  $ r_D$. If $p$ is not yet in $\good$ nor in $\bad$, then for 
all ${-d\in\mathcal L}$ we have ${p\nmid r_d}$ (if $h(-d)=1$, it is because $H_{-d}'=1$ and if $h(-d)>1$ it follows from the step-by-step construction of $\mathcal L$).  
So, if $\chi_d(p)=0$ or $-1$ for some $-d\in\mathcal L$, then we  put $p$ in the list $\good$; else we put it in $\bad$. We add $-D$ to $\mathcal L$ and start again to  (ii) 
(unless $\mathcal L=\mathcal D$).
\end{enumerate}
 
 \noindent\textbf{Results:} We take  $\mathcal D$ to be the list of quadratic imaginary discriminants of  class number in $[\![1,4]\!]$ to which we add $-87$ (see~Appendix) and $N=10^{14}$.  We obtain ($\mathcal L=\mathcal D$ and) $\bad =\varnothing$.  Thus if a prime $11\leq p< 10^{14},\ p\neq 13 $ is such that $\chi_D(p)=0$ or $-1$  for some $-D\in\mathcal D$,  then $X_0^+(p^r)(\Q)$ is trivial for all $r>1.$
 
 \paragraph{ALGORITHM, part II:} In this part, we construct the list $\verybad$ of ``very bad primes'', that is,  the primes $11\leq p<10^{14}$ which split in $\Q(\sqrt{-D})$ for all $-D\in\mathcal D=\mathcal L$. For such primes, we indeed cannot establish the triviality of $X_0^+(p^r)(\Q)$. 
For  this, we refine the  ``trial search'' naive idea as follows.
\begin{enumerate}
 \item We consider   a sublist $\mathcal D' $ of  $\mathcal D$ for which we compute explicitely the values of congruences of primes which split for all $-d\in\mathcal D'$. In practice, we take $$\mathcal D'=\{ -3,-4,-15,-20,-7,-11,-39,-52,-51,-68,-19,-23,-87\} .$$ 
 Since $-4$ and $-3$ are in $\mathcal D'$, a prime $p$ splits for all $-d\in\mathcal D'$ if and  only if $p$ is a non-zero square modulo $q$ for all $q\in\mathcal L'=\{3,4,5,7,11,13,17,19,23,29\}.$  Note that we precisely chose the subset $\mathcal D'$ because, except for $-4$, this is the list of the quadratic imaginary discriminants corresponding to the first nine odd prime numbers. 
 The first twelve discriminants of $\mathcal D'$ are of class number not exceeding~$4$, and $-87$ is of class number~$6$.
 We define 
 $$
 M =
3\times 4\times 5\times 7\times 11\times 13\times 17\times 19\times 23\times  29 =
12\, 939\, 386\, 460.
$$  
There are $1\,†995\, 840$ values of congruences modulo $M$ which are  non-zero squares modulo $q$ for all $q\in\mathcal L'.$  The representatives in the range 
$[\![ 0,M-1]\!]$ of those values make a list $\cal S$. Concretely, to find $\mathcal S,$ we make a list of all non-zero squares modulo $q$ for each $q\in\mathcal L'$ and use the Chinese Remainder Theorem. 
 \item For each value $a\in\mathcal S$  and each integer $p\equiv a \pmod M$ in the range $[\![11,N]\!] ,$ if $p$ is pseudoprime, we test if $\chi_D(p)=1$ for all $-D\in\mathcal D\backslash \mathcal D'$. If it is and  if $p$ is indeed prime we put it in $\verybad$. 
\end{enumerate}

\noindent\textbf{Results:} with $\mathcal D'$ as before and $N=10^{14}$, we obtain $\verybad=\varnothing.$ 

\medskip
The output of this is the following.  
\begin{proposition}
\label{output}
If~$p$ is a prime number, ${11\le p<10^{14}}$ and ${p\ne 13}$, then
$X_0^+ (p^r )(\Q )$ for $r>1$ consist of cusps and CM points. 
\end{proposition}
Together with Theorem~\ref{tse} we obtain Theorem~\ref{MT} of the introduction:
\begin{corollary}
The same conclusion as for Proposition~\ref{output} is true for $X_0^+ (p^r )(\Q )$ with
$p\geq 11$, $p\neq 13$ (and~$r>1$).
\end{corollary}
\begin{proof} {\bf (of Proposition~\ref{output}).} 
By Part II of the algorithm, since  ${\verybad=\varnothing}$, then any prime  $p\geq 11$, $13\neq p<10^{14}$ is inert or ramified in $\Q(\sqrt{-D})$ for 
some $-D\in\mathcal D$. We conclude by Part~I~$(\star)$  since ${\bad=\varnothing}$. 
\qed
\end{proof}

  \begin{remark}\label{verysmall}
\label{Xsplit(13)}

We close this paper by discussing the cursed level 13. As explained in the introduction, the question of the rational points on~$X_0^+ (169)\cong
X_\spl (13)$ is the only remaining open case among the $X_0^+ (p^r )$ for~$r>1$. We do not prove anything new here, but try to use this ``stubbornly 
resisting" example (according to Darmon's expression) to illustrate in details many of the tools used all over the paper.

  First recall that for all prime $p$, the Jacobian~$J_{\mathrm{nonsplit}} (p)$  of the curve
$X_{\mathrm{nonsplit}} (p)$ associated to the normalizer of a nonsplit Cartan subgroup mod $p$ is isomorphic to the newpart $J_0^{+,
\mathrm{new}}(p^2 )$ of the Jacobian $J_0^{+} (p^2 )$ of~$X_0^+ (p^2 )$ (see~\cite{Ch04}). On the other hand, one knows that~$J_0^{+} (p^2 )$
decomposes up to isogeny as
$$
J_0^+ (p^2 ) \sim J_0 ( p) \times J_0^{+, \mathrm{new}}(p^2 ) \sim J_0(p) \times J_{\mathrm{nonsplit}} (p)
$$
(see e.g.~\cite{Mo86}, p. 444). The~$J_0 (p)$ factor in the above decomposition, and more precisely its~$J_0^- (p)$, $J_e (p)$ and $\tilde{J} (p)$
successive subquotients, play a crucial role in our techniques, as they allow to use Mazur's method in order to prove integrality of rational points; as is
well-known, the absence of such quotients is one of the main problems with the case of $X_{\mathrm{nonsplit}} (p)$ or $X_0^+ (p)$.

  Now when $p=13$ one has $J_0 (13)=0$, so the jacobians of $X_{\mathrm{nonsplit}} (p)$ and $X_\spl (p)$ are isogenous. (In
prime level, this is the only case where this interesting phenomenon occurs, as everything is $0$ for $p=2,3,5,7$, i.e. the other $p$'s for
which $g(X_0(p))=0$). Actually more is true: Burcu Baran proved by computing explicit equations that the two above curves are actually
{\it isomorphic} over~$\Q$ (see~\cite{Ba11}). One therefore now faces difficulties of ``nonsplit type''. Our curve is of genus 3, and its jacobian should
be of same rank over~$\Q$, so not only Mazur's method, but also Chabauty's method is of no help here. The thirteen quadratic imaginary
orders with class number one split, according to the decomposition of the number 13 in them, into seven points in $X_{\mathrm{nonsplit}}
(13)(\Q )$ and six points in $X_\spl (13)(\Q )$. (The rational cusp of the latter restores the balance with~$X_{\mathrm{nonsplit}} (13)$).
Galbraith~\cite{Ga02} and Baran~\cite{Ba11} checked there are no rational points but the trivial ones, in a big box (whose size they do not
specify however), but to conclude that there are no point at all we would need some effective Mordell, at least for that particular curve. Our
Theorem~\ref{th1} can still be used as an approximation for {\it integral} points (yielding that their Weil height~$\height (j)$ is bounded
by~$76.4$ - this can be lowered by optimizing the estimations in the proof of Theorem~\ref{th1}), but again we cannot go further by lack of
integrality results... Perhaps the techniques of~\cite{BS10} could be of some help here.
\end{remark}

\section{Appendix : tables and algorithms}

\begin{itemize}
\item\textbf{Quadratic imaginary discriminants  of  class number  in the range $[\![ 1,4 ]\!]  :$}

\noindent Class number $1$: \\
\begin{tabular}{l@{}l}
  $-\{$&$3,4,7,8,11,19,43,67,163\}$
\end{tabular}

\noindent Class number $2$:\\
\begin{tabular}{l@{}l}
 $-\{$&$20,24,40,52,15,88,35,148,51,232,91,115,123,187,235,267,403,427\}$ 
\end{tabular}

\noindent Class number $3$: \\
\begin{tabular}{l@{}l}
$-\{$&$23,31,59,83,107,139,211,283,307,331,379,
499,547,643,883,907\}$
\end{tabular}

\noindent Class number $4$: \\
\begin{tabular}{l@{}l}
$ - \{$&$56,68,84,120,132,136,39,168,184,55,228,280,292,312,328,
340, 372,388,408,520,$\\
&$532,568,
155,708,760,772,195,203,219,1012,259,291,323,355,435,483,555,595,627,$\\
&$667,715,723,763,
795,955,1003,1027,1227,1243,1387,1411,1435,1507,1555\}$  
\end{tabular}

%
%
\item\textbf{Algorithms} :  we reproduce here the pseudo-codes of the algorithms described in~Section~\ref{explsieve}. The original codes have been written with  
Sage~\cite{Sage}. We  used the \textsf{hilbert\_class\_polynomial}  function to compute $H_{-D}$ and the \textsf{crt} function to apply Chinese Remainder Theorem.

\textbf{Algorithm, Part I: } 

\textbf{bad\_discrim\_and\_primes($\mathcal D,N$)}
\begin{algorithmic}[1]
\REQUIRE A list $\mathcal D$ of imaginary quadratic discriminants and an integer $N>1$ (as in Section~\ref{explsieve}).
 \STATE set $L\leftarrow [3,4,7,8,11,19,43,67,163],\quad  \bad\leftarrow [ ],\quad\mbox{and}\quad \good\leftarrow [  ].$
\FOR{$d $ in $ \mathcal D$}
       \STATE \label{boucle} set    $G\leftarrow H_{-d}'$
       \STATE compute the prime factors $\mathcal P_D$ of $ r_D:={\mathrm{gcd}}(\mathrm{Res}(G,H_{-c^2D}'); c\in [\![2,7]\!])$
       \FOR{$p$ in $\mathcal P_D$}
                \IF{$p > 10$ and $p < N+1$ and $p$ not  in $ \good$ and $p$ not in $\bad$}
                          \IF{$\chi_{-m}(p)=1$ for all $m$ in $L$}
                                     \STATE add $p$ to the list $\bad$  
                                   \ELSE 
                                   \STATE add $p$ to the list~$\good.$
                          \ENDIF         
                \ENDIF
       \ENDFOR         
       \STATE add $d$ to the list $L$ (and go to step \ref{boucle} with another $d$ in $\mathcal D$). 
\ENDFOR
                
\RETURN $[L,\bad]$.
\end{algorithmic} 
       
\textbf{Algorithm, Part II: }    
\begin{enumerate}
\item
 Let $\mathcal L'$ be a list of pairwise coprime moduli $d_1,\dots,d_n$ and put ${M=\mathrm{lcm}(\mathcal L')=d_1\dots d_n}$. The following function returns the non-zero squares modulo all the integers $d_1,\dots,d_n$ as a list of the form $[M,[$integers modulo $M]]. $
 
 \textbf{squares\_congruences($\mathcal L'$)}
\begin{algorithmic}[1]
\REQUIRE  a list $\mathcal L'$ of pairwise coprime moduli 
\STATE do a list $[[k^2 \pmod n \mid k\in\{1,\dots,(n-1)/2\}] : n \in \mathcal L']$ 
\RETURN $[\mathrm{lcm}(\mathcal L'),$ Chinese Remainder Theorem applied to the preceeding list$]$.  
      
\end{algorithmic}      
  
\item      Suppose given a list  $C=[M,[s_1,\dots,s_r]]$ with $M$ a moduli and $s_1,\dots,s_r$ integers modulo $M$, a list $L$ of quadratic imaginary discriminants, and two integers $n,m$ with ${n<m}$. The following function gives the prime numbers in range $[n,m[$ which are congruent to some $s_i$ modulo $M$ and which split in all the quadratic fields with discriminant in $L$.

\textbf{very\_bad\_primes($C,L,n,m$)}
\begin{algorithmic}[1]
\REQUIRE  $C,L,n,m$ as before.       
\STATE    $li\leftarrow [ ]$
\FOR{$i\in\{1,\dots,r\}$}
       \STATE $p \leftarrow s_i+\lceil \frac{n-s_i}{M}\rceil *M$ 
           \WHILE{$p < m$}
            \IF{$p$ is pseudoprime and  $\chi_{-D}(p)=1$ for all $D \in L$} 
                \IF{$p$ is prime} 
                     \STATE add $p$ to the list $li$ 
                \ENDIF
             \ENDIF      
            \STATE $p \leftarrow p+ M$
        \ENDWHILE    
\ENDFOR
\RETURN li      
\end{algorithmic}

\end{enumerate} 

\textbf{Applying the algorithms:} 
\begin{algorithmic}[1]
\STATE set $\mathcal D$ to be the list of quadratic imaginary discriminants of class number in range $[\![1,4]\!] $ to which we add $-87$, and\\
set $\mathcal D'=\{ -3,-4,-15,-20,-7,-11,-39,-52,-51,-68,-19,-23,-87 \} .$ 
\STATE $[L,\bad]\leftarrow $ bad\_discrim\_and\_primes($\mathcal D,10^{14}$)
\STATE $C\leftarrow$ square\_congruences($[3,4,5,7,11,13,17,19,23,29]$)
\STATE $V\leftarrow$ very\_bad\_primes($C,\mathcal D\backslash\mathcal D',11,10^{14}$)
\end{algorithmic}
Result: for any prime $p\in [11,10^{14}]$ such that $p\not\in \bad\cup V$, the rational points on $X_0^+(p^r)$ are trivial for all integer $r>1$.

    \end{itemize}

\end{document}